\documentclass[12pt, reqno, a4paper]{amsart}

\usepackage{amssymb}
\usepackage{hyperref}

\usepackage[english]{babel}

\theoremstyle{plain}

\newtheorem*{corollary}{Corollary}
\newtheorem{lemma}{Lemma}
\newtheorem{theorem}{Theorem}
\newtheorem*{conjecture}{Conjecture}

\theoremstyle{remark}
\newtheorem*{remark}{Remark}

\theoremstyle{definition}

\newtheorem{example}{Example}

\DeclareMathOperator{\Id}{Id}

\DeclareMathOperator{\id}{id}
\DeclareMathOperator{\ch}{char}
\DeclareMathOperator{\im}{im}

\DeclareMathOperator{\GL}{GL}

\DeclareMathOperator{\PGL}{PGL}

\DeclareMathOperator{\End}{End}

\DeclareMathOperator{\diag}{diag}

\DeclareMathOperator{\Alt}{Alt}

\DeclareMathOperator{\PIexp}{PIexp}

\begin{document}

\title{Algebras simple with respect to a Taft algebra action}

\author{A.\,S.~Gordienko}
\address{Vrije Universiteit Brussel, Belgium}
\email{alexey.gordienko@vub.ac.be} 


\begin{abstract}
Algebras simple with respect to an action of a Taft algebra $H_{m^2}(\zeta)$
deliver an interesting example of $H$-module algebras that are $H$-simple
but not necessarily semisimple.
We describe finite dimensional $H_{m^2}(\zeta)$-simple algebras and prove the analog of Amitsur's conjecture
for codimensions of their polynomial $H_{m^2}(\zeta)$-identities.
In particular, we show that the Hopf PI-exponent of an $H_{m^2}(\zeta)$-simple algebra $A$ 
over an algebraically closed field of characteristic $0$ equals $\dim A$.
The groups of automorphisms preserving the structure of an $H_{m^2}(\zeta)$-module algebra are studied as well.
\end{abstract}


\thanks{Supported by Fonds Wetenschappelijk Onderzoek~--- Vlaanderen Pegasus Marie Curie post doctoral fellowship (Belgium) and RFBR grant 13-01-00234a (Russia).}

\maketitle

The notion of an $H$-(co)module algebra is a natural generalization of the notion of a graded algebra,
an algebra with an action of a group by automorphisms, and an algebra with an action of a Lie algebra by derivations.
 In particular, if $H_{m^2}(\zeta)$ is the $m^2$-dimensional Taft algebra, an $H_{m^2}(\zeta)$-module
algebra is an algebra endowed both with an action of the cyclic group of order $m$ and with a skew-derivation satisfying certain conditions. The Taft algebra $H_4(-1)$ is called Sweedler's algebra.

The theory of gradings on matrix algebras and simple Lie algebras is a well developed area~\cite{ BahtKochMont, BahturinZaicevSeghalSimpleGraded}. Quaternion $H_4(-1)$-extensions and related crossed products
were considered in~\cite{DoiTakeuchi}. In~\cite{ASGordienko11}, the author classified all
finite dimensional $H_4(-1)$-simple algebras.
Here we classify finite dimensional $H_{m^2}(\zeta)$-simple algebras over an algebraically closed field (Sections~\ref{SectionTaftSimpleSemisimple}--\ref{SectionTaftSimpleNonSemisimple}).

Amitsur's conjecture on asymptotic behaviour of codimensions of ordinary polynomial
identities was proved by A.~Giambruno and M.\,V.~Zaicev~\cite[Theorem~6.5.2]{ZaiGia} in 1999.

Suppose an algebra is endowed with a grading, an action of a group $G$ by automorphisms and anti-automorphisms, an action of a Lie algebra by derivations or a structure of an $H$-module algebra for some Hopf algebra $H$. Then
it is natural to consider, respectively, graded, $G$-, differential or $H$-identities~\cite{BahtGiaZai, BahtZaiGradedExp, BahturinLinchenko,  BereleHopf, Kharchenko}.

The analog of Amitsur's conjecture for polynomial $H$-identities was proved under wide conditions
by the author in~\cite{ASGordienko8, ASGordienko9}. However, in those results the $H$-invariance of the Jacobson radical was required. 
Until now the algebras simple with respect to an action of $H_4(-1)$ were the only example where the analog of Amitsur's conjecture was proved for an $H$-simple
non-semisimple algebra~\cite{ASGordienko11}.
 In this article we prove the analog of Amitsur's conjecture for all finite dimensional $H_{m^2}(\zeta)$-simple algebras not necessarily semisimple (Section~\ref{SectionTaftSimpleAmitsur}) assuming that the base field is algebraically closed and of characteristic $0$.

\section{Introduction}

An algebra $A$
over a field $F$
is an \textit{$H$-module algebra}
for some Hopf algebra $H$
if $A$ is endowed with a homomorphism $H \to \End_F(A)$ such that
$h(ab)=(h_{(1)}a)(h_{(2)}b)$
for all $h \in H$, $a,b \in A$. Here we use Sweedler's notation
$\Delta h = h_{(1)} \otimes h_{(2)}$ where $\Delta$ is the comultiplication
in $H$.
We refer the reader to~\cite{Danara, Montgomery, Sweedler}
   for an account
  of Hopf algebras and algebras with Hopf algebra actions.

Let $A$ be an $H$-module algebra for some Hopf algebra $H$ over a field $F$.
We say that $A$ is \textit{$H$-simple} if $A^2\ne 0$ and $A$ has no non-trivial
two-sided $H$-invariant ideals. 

Let $m \geqslant 2$ be an integer and let $\zeta$ be a primitive $m$th root of unity
in a field $F$. (Such root exists in $F$ only if $\ch F \nmid m$.)
Consider the algebra $H_{m^2}(\zeta)$ with unity generated
by elements $c$ and $v$ satisfying the relations $c^m=1$, $v^m=0$, $vc=\zeta cv$.
Note that $(c^i v^k)_{0 \leqslant i, k \leqslant m-1}$ is a basis of $H_{m^2}(\zeta)$.
We introduce on $H_{m^2}(\zeta)$ a structure of a coalgebra by
  $\Delta(c)=c\otimes c$,
$\Delta(v) = c\otimes v + v\otimes 1$, $\varepsilon(c)=1$, $\varepsilon(v)=0$.
Then $H_{m^2}(\zeta)$ is a Hopf algebra with the antipode $S$ where $S(c)=c^{-1}$
and $S(v)=-c^{-1}v$. The algebra $H_{m^2}(\zeta)$ is called a \textit{Taft algebra}.

\begin{remark}
Note that if $A$ is an $H_{m^2}(\zeta)$-module algebra, then the group $\langle c \rangle \cong \mathbb Z_m$
is acting on $A$ by automorphisms. Every algebra $A$ with a $\mathbb Z_m$-action by automorphisms
is a $\mathbb Z_m$-graded algebra: $$A^{(i)} = \lbrace a \in A \mid ca = \zeta^i a\rbrace,$$
$A^{(i)}A^{(k)}\subseteq A^{(i+k)}$. Conversely, if $A = \bigoplus_{i=0}^{m-1} A^{(i)}$ is a $\mathbb Z_m$-graded algebra,
then $\mathbb Z_m$ is acting on $A$ by automorphisms: $c a^{(i)} = \zeta^i a^{(i)}$
for all $a^{(i)} \in A^{(i)}$. Moreover, the notions of $\mathbb Z_m$-simple and simple $\mathbb Z_m$-graded algebras are equivalent.
\end{remark}

\begin{remark}
\cite[Theorems~5 and~6]{BahturinZaicevSeghalGroupGrAssoc} imply that every $\mathbb Z_m$-grading
on $M_n(F)$, where $F$ is an algebraically closed field, is, up to a conjugation, \textit{elementary}, i.e. there exist $g_1, g_2, \ldots, g_n \in \mathbb Z_m$ such that each matrix unit $e_{ij}$ belongs to $A^{(g_i^{-1} g_j)}$. Rearranging rows and columns, we may assume that every $\mathbb Z_m$-action on
$M_n(F)$ is defined by $c a = Q^{-1} a Q$ for some matrix $$Q=\diag\{\underbrace{1, \ldots, 1}_{k_0},
\underbrace{\zeta, \ldots, \zeta}_{k_1}, \ldots, \underbrace{\zeta^{m-1}, \ldots, \zeta^{m-1}}_{k_{m-1}}\}.$$
\end{remark}

\section{Semisimple $H_{m^2}(\zeta)$-simple algebras}\label{SectionTaftSimpleSemisimple}

In this section we treat the case when an $H_{m^2}(\zeta)$-simple algebra $A$ is semisimple.

\begin{theorem}\label{TheoremTaftSimpleSemisimple}
Let $A$ be a semisimple $H_{m^2}(\zeta)$-simple algebra over an algebraically closed field~$F$. Then
 $$A \cong \underbrace{M_k(F) \oplus M_k(F) \oplus \dots \oplus M_k(F)}_t\qquad \text{(direct sum of ideals)}$$ for some $k,t\in\mathbb N$, $t \mid m$,
and there exist $P \in M_k(F)$ and $Q \in \GL_k(F)$  where  $Q^{\frac{m}{t}} = E_k$, $E_k$ is the identity matrix $k\times k$, $Q P Q^{-1}=\zeta^{-t} P$, $P^m = \alpha E_k$ for some $\alpha \in F$, such that
\begin{equation}\label{EqSSTaftSimple1} c\, (a_1, a_2, \ldots, a_t) = (Q a_t Q^{-1}, a_1, \ldots, a_{t-1}),
\end{equation}\begin{equation}\label{EqSSTaftSimple2} v\,(a_1, a_2, \ldots, a_t)=(Pa_1 - (Q a_t Q^{-1}) P, \zeta (P a_2 - a_1 P), \ldots, \zeta^{t-1}(Pa_t - a_{t-1} P))\end{equation}
for all $a_1, a_2, \ldots, a_t \in M_k(F)$.
\end{theorem}

\begin{remark}
Diagonalizing $Q$, we may assume that $$Q = \diag\lbrace\underbrace{1,\ldots,1}_{k_1},
\underbrace{\zeta^t,\ldots,\zeta^t}_{k_2}, \dots, \underbrace{\zeta^{t\left(\frac{m}{t}-1\right)},\ldots,\zeta^{t\left(\frac{m}{t}-1\right)}}_{k_{\frac{m}{t}}}\rbrace$$ 
for some $k_1,\ldots,k_{\frac{m}{t}}\in \mathbb Z_+$, $k_1+\ldots+k_{\frac{m}{t}}=k$.
Now $QPQ^{-1} = \zeta^{-1} P$ imply
that $P=(P_{ij})$ is a block matrix where $P_{ij}$ is an matrix $k_{i-1} \times k_{j-1}$ and
$P_{ij} = 0$ for all $j \ne i+1$ and $(i,j)\ne \left(\frac{m}{t}, 1\right)$.
\end{remark}

We begin with three auxiliary lemmas.  In the first two, we prove all the assertions of Theorem~\ref{TheoremTaftSimpleSemisimple} except $P^m=\alpha E_k$. In Lemma~\ref{LemmaTaftSimpleMatrix}
we treat the case when $A$ isomorphic to a full matrix algebra.

\begin{lemma}\label{LemmaTaftSimpleMatrix}
Let $A$ be an $H_{m^2}(\zeta)$-module algebra over an algebraically closed field $F$, isomorphic as an algebra to $M_k(F)$ for some $k\in \mathbb N$.
Then there exist matrices $P\in M_k(F)$, $Q\in \GL_k(F)$, $Q^m=E_k$  such that $Q P Q^{-1}=\zeta^{-1} P$ and
$A$ is isomorphic as an $H_{m^2}(\zeta)$-module algebra to $M_k(F)$ with the following
$H_{m^2}(\zeta)$-action: $ca=QaQ^{-1}$ and $va=Pa-(QaQ^{-1})P$ for all $a\in M_k(F)$.
\end{lemma}
\begin{proof} All automorphisms of full matrix algebras are inner. Hence $ca=QaQ^{-1}$
for some  $Q\in \GL_k(F)$. Since $c^m = 1$, the matrix $Q^m$ is scalar. Multiplying $Q$ by the $m$th root
of the corresponding scalar, we may assume that $Q^m = E_k$.

 Recall that $v$ is acting on $A$ by a skew-derivation.
We claim\footnote{This result is a ``folklore'' one. I am grateful to V.\,K.~Kharchenko who informed me of a simple proof of it.} that this skew-derivation is \textit{inner},  i.e. there exists a matrix $P \in A$ such that
$va = Pa-(ca)P$ for all $a\in A$. Indeed, $$Q^{-1}(v(ab)) = Q^{-1}((ca)(vb)+(va)b)
= Q^{-1}((QaQ^{-1})(vb)+(va)b)=a(Q^{-1}(vb))+(Q^{-1}(va))b$$
for all $a,b \in A$. Hence $Q^{-1}(v(\cdot))$
is a derivation and $Q^{-1}(va)=P_0a-aP_0$
for all $a\in A$ for some $P_0 \in A$.
 Thus $$va = QP_0a- Q a P_0= QP_0 a - Q aQ^{-1}QP_0= Pa-(QaQ^{-1})P
 \text{ for all }a\in A$$ where $P=QP_0$, i.e. $v$ acts as an inner skew-derivation.
 
 Note that $vc=\zeta cv$ implies $c^{-1}v=\zeta v c^{-1}$,
$$Q^{-1}(Pa-(QaQ^{-1})P)Q = \zeta P(Q^{-1}aQ)-aP,$$ 
$$Q^{-1}PaQ-aQ^{-1}PQ = \zeta PQ^{-1}aQ-\zeta aP,$$ 
$$Q^{-1}Pa-aQ^{-1}P = \zeta PQ^{-1}a-\zeta aPQ^{-1},$$ 
$$Q^{-1}Pa-\zeta PQ^{-1}a = aQ^{-1}P-\zeta aPQ^{-1},$$ 
$$(Q^{-1}P-\zeta PQ^{-1})a = a(Q^{-1}P-\zeta PQ^{-1}) \text{ for all } a\in A.$$
Hence $Q^{-1}P-\zeta PQ^{-1} = \alpha E_k$ for some $\alpha \in F$. Now we replace $P$ with $(P-\frac{\alpha}{1-\zeta} Q)$.
Then $v$ is the same but $Q^{-1}P-\zeta PQ^{-1} = 0$ and $Q PQ^{-1} = \zeta^{-1} P$.
\end{proof}

Here we treat the general case.

\begin{lemma}\label{LemmaTaftSimpleSemisimpleFirst}
Let $A$ be a semisimple $H_{m^2}(\zeta)$-simple algebra over an algebraically closed field $F$. Then
 $A \cong \underbrace{M_k(F) \oplus M_k(F) \oplus \dots \oplus M_k(F)}_t$ (direct sum of ideals) for some $k, t \in\mathbb N$, $t \mid m$,
and there exist $P \in M_k(F)$ and $Q \in \GL_k(F)$, $Q^{\frac{m}{t}} = E_k$, $Q P Q^{-1}=\zeta^{-t} P$, such that~(\ref{EqSSTaftSimple1}) and~(\ref{EqSSTaftSimple2}) hold for all $a_1, a_2, \ldots, a_t \in M_k(F)$.
\end{lemma}
\begin{proof}
If $A$ is semisimple, then $A$ is the direct sum of $\mathbb Z_m$-simple subalgebras.
Let $B$ be one of such subalgebras.
Then $vb= v(1_B b)=(c1_B)(vb)+(v1_B)b \in B$
for all $b\in B$. Hence $B$ is an $H_{m^2}(\zeta)$-submodule, $A=B$,
and $A$ is a $\mathbb Z_m$-simple algebra.
Therefore,  $A \cong \underbrace{M_k(F) \oplus M_k(F) \oplus \dots \oplus M_k(F)}_t$ (direct sum of ideals) for some $k,t \in\mathbb N$, $t \mid m$,
and $c$ maps the $i$th component to the $(i+1)$th. 

In the case $t=1$, the assertion is proved in Lemma~\ref{LemmaTaftSimpleMatrix}. 
Consider the case $t\geqslant 2$.
Note that $c^t$ maps each component onto itself.
Since every automorphism of the matrix algebra is inner,
there exist $Q$ such that $c^t(a, 0, \ldots, 0) = (QaQ^{-1}, 0, \ldots, 0)$
for any $a \in M_k(F)$. Now $c^m = \id_A$ implies that $Q^{\frac{m}{t}}$ is a scalar matrix
and we may assume that $Q^{\frac{m}{t}} = E_k$ since the field $F$ is algebraically closed
and we can multiply $Q$ by the $m$th root of the corresponding scalar.
Therefore, we may assume that~(\ref{EqSSTaftSimple1}) holds.

Let $\pi_i \colon A \to M_k(F)$ be the natural projections on the $i$th component. Consider
 $\rho_{ij} \in \End_F(M_k(F))$, $1 \leqslant i,j\leqslant t$,
defined by $\rho_{ij}(a):= \pi_i(v\,(\underbrace{0, \ldots, 0}_{j-1}, a,0, \ldots, 0))$ for $a \in M_k(F)$.
Then \begin{equation*}\begin{split}\rho_{ij}(ab)=\pi_i(v\,(0,\ldots, 0, ab,0, \ldots, 0))=\\
\pi_i(v((0,\ldots, 0, a,0, \ldots, 0)(0,\ldots, 0, b,0, \ldots, 0)))=\\
\pi_i((c(0,\ldots, 0, a,0, \ldots, 0))v(0,\ldots, 0, b,0, \ldots, 0))+ \\
\pi_i((v(0,\ldots, 0, a,0, \ldots, 0))(0,\ldots, 0, b,0, \ldots, 0))= \\
\delta_{ij}\,\rho_{ii}(a)b+\delta_{j, i-1}\,a\rho_{i,i-1}(b)+\delta_{i1}\delta_{jt}\,QaQ^{-1}\rho_{1t}(b)
\end{split}\end{equation*}
for all $a,b \in M_k(F)$ where $\delta_{ij}$ is the Kronecker delta.

Let $\rho_{ii}(E_k)=P_i$, $\rho_{i, i-1}(E_k)=Q_i$, $\rho_{1t}(E_k)=Q_1$ where
$P_i, Q_i \in M_k(F)$.
Then \begin{equation*}\begin{split}v(a_1, \ldots, a_m)=(\pi_1(v\,(a_1,\ldots, a_t)), \ldots, \pi_1(v\,(a_1,\ldots, a_t)))= \\
(P_1 a_1 + (Qa_tQ^{-1}) Q_1, P_2 a_2 + a_1 Q_2, \ldots, P_t a_t + a_{t-1} Q_t).
\end{split}\end{equation*}
Now we notice that \begin{equation*}\begin{split}0=v((\underbrace{0, \ldots, 0}_{i-1}, E_k,0, \ldots, 0)(\underbrace{0, \ldots, 0}_i, E_k,0, \ldots, 0))=\\ (c(\underbrace{0, \ldots, 0}_{i-1}, E_k,0, \ldots, 0))(v(\underbrace{0, \ldots, 0}_i, E_k,0, \ldots, 0))+\\ (v(\underbrace{0, \ldots, 0}_{i-1}, E_k,0, \ldots, 0))(\underbrace{0, \ldots, 0}_i, E_k,0, \ldots, 0)=\\ (\underbrace{0, \ldots, 0}_i, P_{i+1}+Q_{i+1}, 0, \ldots, 0).\end{split}\end{equation*}
Thus $Q_{i+1}=-P_{i+1}$.

Note that
\begin{equation*}\begin{split}(-\zeta QP_tQ^{-1}, 0,  \ldots, 0, \zeta P_{t-1})=\zeta cv(0, \ldots, 0, E_k, 0)=vc(0, \ldots, 0, E_k, 0)=\\ v(0, \ldots, 0, E_k)=(-P_1,0, \ldots, 0, P_t),\end{split}\end{equation*}
\begin{equation*}\begin{split}(\zeta QP_tQ^{-1}, -\zeta P_1,0, \ldots, 0)=\zeta cv(0, \ldots, 0, E_k)=vc(0, \ldots, 0, E_k)=\\ v(E_k,0, \ldots, 0)=(P_1, -P_2,0, \ldots, 0),\end{split}\end{equation*}
and $P_1 = \zeta Q P_t Q^{-1}$, $P_2 = \zeta P_1$, $P_t = \zeta P_{t-1}$.

Moreover, if $t > 2$,  \begin{equation*}\begin{split}(\underbrace{0, \ldots, 0}_i, \zeta P_i, -\zeta P_{i+1},0, \ldots, 0)=\zeta cv(\underbrace{0, \ldots, 0}_{i-1}, E_k,0, \ldots, 0)=\\ vc(\underbrace{0, \ldots, 0}_{i-1}, E_k,0, \ldots, 0)=v(\underbrace{0, \ldots, 0}_i, E_k,0, \ldots, 0)=(\underbrace{0, \ldots, 0}_i, P_{i+1}, -P_{i+2},0, \ldots, 0)\end{split}\end{equation*}
for $1\leqslant i \leqslant t-2$.
Therefore, $P_{i+1}=\zeta P_i$ for all $1\leqslant i \leqslant t-1$.
Let $P:= P_1$. Then $P_i = \zeta^{i-1} P$, $\zeta^t QP Q^{-1} = P$,  (\ref{EqSSTaftSimple2}) holds, and the lemma is proved.
\end{proof}

Recall the definition of \textit{quantum binomial coefficients}: $$\binom{n}{k}_\zeta := \frac{n!_\zeta}{(n-k)!_\zeta\ k!_\zeta}$$ where $n!_\zeta := n_\zeta (n-1)_\zeta \cdot \dots \cdot 1_\zeta$ and
$n_\zeta := 1 + \zeta + \zeta^2 + \dots + \zeta^{n-1}$.

\begin{lemma}\label{LemmaTaftSimpleSemisimpleFormula}
Let $v$ be the operator defined on $M_k(F)^t$ by~(\ref{EqSSTaftSimple2}) where $QPQ^{-1}=\zeta^{-t} P$.
Then $$v^\ell (a_1, a_2, \ldots, a_t) =(b_1, b_2, \ldots, b_t)$$
where \begin{equation}\label{EqSSTaftSimple3} b_k = \zeta^{\ell(k-1)} \sum_{j=0}^\ell  (-1)^j \zeta^{-\frac{j(j-1)}{2}} \binom{\ell}{j}_{\zeta^{-1}}
P^{\ell-j} a_{k-j} P^j \end{equation}
and $a_{-j} := Q a_{t-j} Q^{-1}$, $j \geqslant 0$, $a_i \in M_k(F)$, $1\leqslant \ell \leqslant m$.
\end{lemma}
\begin{proof}
We prove the assertion by induction on $\ell$. The base $\ell=1$ is evident.
Suppose~(\ref{EqSSTaftSimple3}) holds for $\ell$. Then
$$v^{\ell+1} (a_1, a_2, \ldots, a_t) =(\tilde b_1, \tilde b_2, \ldots, \tilde b_t)$$
where $\tilde b_k = \zeta^{k-1}(P b_k - b_{k-1} P)$, $1\leqslant k \leqslant t$, and $b_0 := Q b_t Q^{-1}$.
Then
\begin{equation*}\begin{split}
\tilde b_k =  \zeta^{k-1}\left( \zeta^{\ell(k-1)} \sum_{j=0}^\ell (-1)^j \zeta^{-\frac{j(j-1)}{2}} \binom{\ell}{j}_{\zeta^{-1}}
P^{\ell-j+1} a_{k-j} P^j  -\right. \\ \left. \zeta^{\ell(k-2)} \sum_{j=0}^\ell (-1)^j \zeta^{-\frac{j(j-1)}{2}} \binom{\ell}{j}_{\zeta^{-1}}
P^{\ell-j} a_{k-j-1} P^{j+1} \right)= \\ \zeta^{k-1}\left( \zeta^{\ell(k-1)} \sum_{j=0}^\ell (-1)^j \zeta^{-\frac{j(j-1)}{2}} \binom{\ell}{j}_{\zeta^{-1}}
P^{\ell-j+1} a_{k-j} P^j  -\right. \\ \left. \zeta^{\ell(k-2)} \sum_{j=1}^{\ell+1} (-1)^{j-1} \zeta^{-\frac{(j-2)(j-1)}{2}} \binom{\ell}{j-1}_{\zeta^{-1}}
P^{\ell-j+1} a_{k-j} P^j \right)= \\ \zeta^{(\ell+1)(k-1)}\left( \sum_{j=0}^\ell (-1)^j \zeta^{-\frac{j(j-1)}{2}} \binom{\ell}{j}_{\zeta^{-1}}
P^{\ell-j+1} a_{k-j} P^j  +\right. \\ \left.  \sum_{j=1}^{\ell+1} (-1)^j \zeta^{-\frac{j(j-1)}{2}} \zeta^{j-\ell-1} \binom{\ell}{j-1}_{\zeta^{-1}}
P^{\ell-j+1} a_{k-j} P^j \right)=\\
\zeta^{(\ell+1)(k-1)} \sum_{j=0}^{\ell+1} (-1)^j \zeta^{-\frac{j(j-1)}{2}} \binom{\ell+1}{j}_{\zeta^{-1}}
P^{\ell-j+1} a_{k-j} P^j.\end{split}\end{equation*}  Therefore, (\ref{EqSSTaftSimple3}) holds for every $1\leqslant \ell \leqslant m$.
\end{proof}
\begin{proof}[Proof of Theorem~\ref{TheoremTaftSimpleSemisimple}]
Recall that $v^m=0$ and $\binom{m}{j}_{\zeta^{-1}}=0$ for $1\leqslant j \leqslant m-1$. Thus Lemmas~\ref{LemmaTaftSimpleSemisimpleFirst}
and \ref{LemmaTaftSimpleSemisimpleFormula} imply
\begin{equation*}\begin{split}v^m(a_1, \ldots, a_t)=(P^m a_1 - a_{1-m}P^m, P^m a_2 - a_{2-m}P^m, \ldots, P^m a_t - a_{t-m}P^m)=\\
([P^m, a_1], [P^m, a_2], \ldots, [P^m, a_t]) = 0\end{split}\end{equation*}
for all $a_i \in M_k(F)$ since $Q^{\frac{m}{t}}=E_k$. Therefore, $P^m = \alpha E_k$ for some $\alpha\in F$,
and we get the theorem.
\end{proof}

\begin{remark}
Conversely, for every $k, t \in\mathbb N$, $t \mid m$, and matrices  $P \in M_k(F)$ and $Q \in \GL_k(F)$
such that $Q^{\frac{m}{t}} = E_k$, $Q P Q^{-1}=\zeta^{-t} P$, $P^m = \alpha E_k$ for some $\alpha \in F$,
we can define the structure of an $H_{m^2}(\zeta)$-simple algebra on 
$A \cong \underbrace{M_k(F) \oplus M_k(F) \oplus \dots \oplus M_k(F)}_t$ (direct sum of ideals) by~(\ref{EqSSTaftSimple1}) and~(\ref{EqSSTaftSimple2}), and this algebra $A$ is even $\mathbb Z_m$-simple.
\end{remark}

\begin{theorem}\label{TheoremTaftSimpleSSIso}
Let $A \cong \underbrace{M_k(F) \oplus M_k(F) \oplus \dots \oplus M_k(F)}_t$ (direct sum of ideals) be a semisimple $H_{m^2}(\zeta)$-simple algebra over a field $F$
defined by matrices $P_1 \in M_k(F)$ and $Q_1 \in \GL_k(F)$
 by~(\ref{EqSSTaftSimple1}) and~(\ref{EqSSTaftSimple2}),
and let $A_2$ be another such algebra defined by matrices $P_2 \in M_k(F)$ and $Q_2 \in \GL_k(F)$.
Then $A_1 \cong A_2$ as algebras and $H_{m^2}(\zeta)$-modules
if and only if $P_2 =\zeta^r\, T P_1 T^{-1}$ and $Q_2 = \beta T Q_1 T^{-1}$ for some $r\in\mathbb Z$, $\beta \in F$, and $T \in \GL_k(F)$.
\end{theorem}
\begin{proof}
Note that in each of $A_1$ and $A_2$ there exist exactly $t$ simple ideals 
isomorphic to $M_k(F)$. Moreover, each isomorphism of $M_k(F)$ is inner.
Therefore, if $\varphi \colon A_1 \to A_2$ is an isomorphism
of algebras and $H_{m^2}(\zeta)$-modules, then there exist 
matrices $T_i \in \GL_k(F)$ and a number $0 \leqslant r < t$ such that
\begin{equation*}\begin{split}\varphi(a_1, \ldots, a_m)=(T_{r+1} a_{r+1} T_{r+1}^{-1}, T_{r+2} a_{r+2} T_{r+2}^{-1},
\ldots, T_t a_t T_t^{-1},\\ T_1 a_1 T_1^{-1}, T_2 a_2 T_2^{-1}, \ldots, T_r a_r T_r^{-1})\end{split}\end{equation*}
for all $a_i \in M_k(F)$. (Here we use the fact that $\varphi$ must commute with $c$.)
Using $c\varphi = \varphi c$ once again, we get
\begin{equation*}\begin{split}c\varphi(a_1, \ldots, a_m)=(Q_2 T_r a_r T_r^{-1} Q_2^{-1}, T_{r+1} a_{r+1} T_{r+1}^{-1}, T_{r+2} a_{r+2} T_{r+2}^{-1},
\ldots, T_t a_t T_t^{-1},\\ T_1 a_1 T_1^{-1}, T_2 a_2 T_2^{-1}, \ldots, T_{r-1} a_{r-1} T_{r-1}^{-1})=\\ \varphi c(a_1, \ldots, a_m)=(T_{r+1} a_r T_{r+1}^{-1}, T_{r+2} a_{r+1} T_{r+2}^{-1},
\ldots, T_t a_{t-1} T_t^{-1},\\ T_1 Q_1 a_t Q_1^{-1} T_1^{-1}, T_2 a_1 T_2^{-1}, \ldots, T_r a_{r-1} T_r^{-1})\end{split}\end{equation*}
for all $a_i \in M_k(F)$.
Therefore, $T_i$ is proportional to $T_{i+1}$ for $1 \leqslant i \leqslant r-1$, $r+1\leqslant i \leqslant t-1$. In addition, $Q_2 T_r$ is proportional to $T_{r+1}$, and $T_t$ is proportional to $T_1 Q_1$.
Multiplying $T_i$ by scalars, we may assume that $T_1 = \ldots = T_r$, $T_{r+1}=\ldots=T_t = T_1 Q_1$.  Let $T := T_{r+1}$.
Then \begin{equation*}\begin{split}\varphi(a_1, \ldots, a_m)=(T a_{r+1} T^{-1}, T a_{r+2} T^{-1},
\ldots, T a_t T^{-1},\\  TQ_1^{-1} a_1 Q_1 T^{-1},  T Q_1^{-1} a_2 Q_1 T^{-1}, \ldots, TQ_1^{-1} a_r Q_1T^{-1} )\end{split}\end{equation*}
and $Q_2 = \beta T Q_1 T^{-1}$ for some $\beta \in F$.

Using $v\varphi = \varphi v$, we get
\begin{equation*}\begin{split}(P_2 T a_{r+1} T^{-1} - T a_r T^{-1} P_2, 
\zeta (P_2 T a_{r+2} T^{-1} - T a_{r+1} T^{-1} P_2), 
\ldots,\\ \zeta^{t-r-1}(P_2 T a_t T^{-1} - T a_{t-1} T^{-1}  P_2),\\ 
\zeta^{t-r}(P_2 T  Q_1^{-1}  a_1  Q_1 T^{-1}- T a_t T^{-1} P_2), \zeta^{t-r+1}(P_2 T  Q_1^{-1}  a_2  Q_1  T^{-1} - T  Q_1^{-1}  a_1  Q_1 T^{-1} P_2), \ldots,\\ \zeta^{t-1}(P_2 T  Q_1^{-1} a_r  Q_1  T^{-1}- T  Q_1^{-1} a_{r-1}  Q_1 T^{-1}P_2 ))=\\
(P_2 T a_{r+1} T^{-1} - Q_2 T Q_1^{-1} a_r Q_1 T^{-1} Q_2^{-1} P_2, 
\zeta (P_2 T a_{r+2} T^{-1} - T a_{r+1} T^{-1} P_2), 
\ldots,\\ \zeta^{t-r-1}(P_2 T a_t T^{-1} - T a_{t-1} T^{-1}  P_2),\\ 
\zeta^{t-r}(P_2 T  Q_1^{-1}  a_1  Q_1 T^{-1}- T a_t T^{-1} P_2), \zeta^{t-r+1}(P_2 T  Q_1^{-1}  a_2  Q_1  T^{-1} - T  Q_1^{-1}  a_1  Q_1 T^{-1} P_2), \ldots,\\ \zeta^{t-1}(P_2 T  Q_1^{-1} a_r  Q_1  T^{-1}- T  Q_1^{-1} a_{r-1}  Q_1 T^{-1}P_2 ))=\\ v(T a_{r+1} T^{-1}, T a_{r+2} T^{-1},
\ldots, T a_t T^{-1},\\  TQ_1^{-1} a_1 Q_1 T^{-1},  T Q_1^{-1} a_2 Q_1 T^{-1}, \ldots, TQ_1^{-1} a_r Q_1T^{-1})=v\varphi(a_1, \ldots, a_m)=\\ \varphi v(a_1, \ldots, a_m)
=\varphi (P_1 a_1 - (Q_1 a_t Q_1^{-1}) P_1, \zeta (P_1 a_2 - a_1 P_1),  \ldots,  \zeta^{t-1}(P_1 a_t - a_{t-1} P_1))=\\
(\zeta^r T (P_1 a_{r+1} - a_r P_1) T^{-1},  \zeta^{r+1} T (P_1 a_{r+2} - a_{r+1} P_1) T^{-1}, \ldots,  \zeta^{t-1} T (P_1 a_t - a_{t-1} P_1) T^{-1}, \\ TQ_1^{-1}(P_1 a_1 - (Q_1 a_t Q_1^{-1}) P_1)Q_1 T^{-1},  \zeta TQ_1^{-1}(P_1 a_2 - a_1 P_1)Q_1 T^{-1}, \ldots,\\ \zeta^{r-1}T Q_1^{-1}(P_1 a_r - a_{r-1} P_1)Q_1 T^{-1})=\\
(\zeta^r (T P_1 a_{r+1} T^{-1} - T a_r P_1 T^{-1}),  \zeta^{r+1} (T P_1 a_{r+2}  T^{-1} - T a_{r+1} P_1 T^{-1}), \ldots,\\  \zeta^{t-1} (T P_1 a_t  T^{-1} - T a_{t-1} P_1 T^{-1}), \\ TQ_1^{-1} P_1 a_1 Q_1 T^{-1} - T a_t (Q_1^{-1} P_1 Q_1) T^{-1},  \zeta (TQ_1^{-1} P_1 a_2 Q_1 T^{-1} - T Q_1^{-1} a_1 P_1 Q_1 T^{-1}), \ldots,\\ \zeta^{r-1}(T Q_1^{-1}P_1 a_r Q_1 T^{-1} - T Q_1^{-1} a_{r-1} P_1 Q_1 T^{-1}))\end{split}\end{equation*}
for all $a_i \in M_k(F)$. Hence $$P_2 = \zeta^r\, T P_1 T^{-1} = \zeta^{r-t}\, T Q_1^{-1} P_1 Q_1 T^{-1} $$ if $r > 0$,
and $P_2 = T P_1 T^{-1}$ if $r=0$. Taking $Q_1 P_1 Q_1^{-1} = \zeta^{-t} P_1$ into account, we reduce both conditions to
$P_2 =\zeta^r\, T P_1 T^{-1}$.

The converse assertion is proved explicitly. If $P_2 =\zeta^r\, T P_1 T^{-1}$ for some $r \in \mathbb Z$,  we can always make $0 \leqslant r < t$ conjugating $P_1$ by $Q_1$.
\end{proof}

\begin{remark}
Therefore every automorphism of a semisimple $H_{m^2}(\zeta)$-simple algebra $A$
that corresponds to a number $t \in \mathbb N$, and matrices $P \in M_k(F)$, $Q \in \GL_k(F)$, can be identified with a pair $(\bar T, r)$, $0\leqslant r < t$ where $T \in \GL_k(F)$, $QTQ^{-1}T^{-1} = \beta E_k$ for some $\beta \in F$, $P =\zeta^r\, T P T^{-1}$.
(Here by $\bar T$ we denote the class of a matrix $T \in \GL_k(F)$ in $\PGL_k(F)$.)
If we transfer the multiplication from the automorphism group to the set of such pairs,
we get $$(\overline W,s)(\overline T,r)=\left\lbrace\begin{array}{llcc} (\overline{WT}, & r+s) & \text{ if }& r+s < t, \\
(\overline{WTQ^{-1}}, & r+s-t) & \text{ if }& r+s \geqslant t. \end{array}\right.$$
Therefore, the automorphism group of $A$ is an extension of a subgroup of $\mathbb Z_m$ by a subgroup of $\PGL_k(F)$.
\end{remark}

\begin{remark}
The case $m=2$ is worked out in detail in~\cite{ASGordienko11}.
Below we list several examples that are consequences of Theorems~\ref{TheoremTaftSimpleSemisimple} and~\ref{TheoremTaftSimpleSSIso}.
\end{remark}

 \begin{example}
 In the case of $m=2$ and $A\cong M_2(F)$ we have the following variants:
 \begin{enumerate}
 \item $A = A^{(0)}= M_2(F)$, $A^{(1)}= 0$, $ca=a$, $va=0$ for all $a\in A$;
 \item $A = A^{(0)} \oplus A^{(1)}$ where $$A^{(0)}=\left\lbrace\left(\begin{array}{cc}
\alpha & 0 \\
0      & \beta 
 \end{array}\right) \mathbin{\biggl|} \alpha,\beta \in F\right\rbrace$$
 and $$A^{(1)}=\left\lbrace\left(\begin{array}{cc}
0 & \alpha \\
\beta      & 0 
 \end{array}\right) \mathbin{\biggl|} \alpha,\beta \in F\right\rbrace,$$
 $ca=(-1)^{i}a$, $va=0$ for $a\in A^{(i)}$;
  \item $A = A^{(0)} \oplus A^{(1)}$ where $$A^{(0)}=\left\lbrace\left(\begin{array}{cc}
\alpha & 0 \\
0      & \beta 
 \end{array}\right) \mathbin{\biggl|} \alpha,\beta \in F\right\rbrace$$
 and $$A^{(1)}=\left\lbrace\left(\begin{array}{cc}
0 & \alpha \\
\beta      & 0 
 \end{array}\right) \mathbin{\biggl|} \alpha,\beta \in F\right\rbrace,$$
 $ca=(-1)^{i}a$, $va = Pa-(ca)P$ for $a\in A^{(i)}$ where $P=\left(\begin{array}{cc}
0 & 1 \\
\gamma      & 0 
 \end{array}\right)$ and $\gamma \in F$ is a fixed number.
  \end{enumerate}
 \end{example}

\begin{example}
Every semisimple $H_4(-1)$-simple algebra $A$ over an algebraically closed field $F$,
$\ch F \ne 2$, that is not simple as an ordinary algebra,
is isomorphic to
$M_k(F) \oplus M_k(F)$ (direct sum of ideals) for some $k \geqslant 1$
where
$$ c\, (a, b) = (b,a),\qquad v\,(a,b)=(Pa-bP,aP-Pb)$$
for all $a,b \in M_k(F)$
and
\begin{enumerate}
\item
either $P=(\underbrace{\alpha,\alpha,\ldots, \alpha}_{k_1}, 
\underbrace{-\alpha,-\alpha,\ldots, -\alpha}_{k_2})$ for some $\alpha \in F$
and $k_1 \geqslant k_2$, $k_1+k_2=k$,
\item or $P$ is a block diagonal matrix with several blocks $\left(\begin{smallmatrix} 
 0 & 1 \\
 0 & 0\\
 \end{smallmatrix}\right)$ on the main diagonal (the rest cells are filled with zeros)
 \end{enumerate}
and these algebras are not isomorphic for different $P$.
\end{example}

\section{Non-semisimple algebras}\label{SectionTaftSimpleNonSemisimple}

First we construct an example of an $H_{m^2}(\zeta)$-simple algebra and then we prove that every non-semisimple $H_{m^2}(\zeta)$-simple algebra is isomorphic to one of the algebras below.

\begin{theorem}\label{TheoremTaftSimpleNonSemiSimplePresent} Let $B$ be a simple $\mathbb Z_m$-graded algebra over a field $F$. Suppose $F$ contains some primitive $m$th root of unity $\zeta$. Define $\mathbb Z_m$-graded vector spaces $W_i$, $1\leqslant i \leqslant m-1$, $W_0 := B$, with linear isomorphisms $\varphi \colon W_{i-1} \to W_i$ (we denote the isomorphisms by the same letter), $1 \leqslant i \leqslant m-1$, such that $\varphi(W_{i-1}^{(\ell)})=W_i^{(\ell+1)}$. Let $\varphi(W_{m-1})=0$.
Consider $H_{m^2}(\zeta)$-module $A=\bigoplus_{i=0}^{m-1} W_i$ (direct sum of subspaces) where $v\varphi(a)=a$ for all $a \in W_i$,
$0\leqslant i \leqslant m-2$, $vB=0$,
and $c a^{(i)}=\zeta^i a^{(i)}$,  $a^{(i)} \in A^{(i)}$, $A^{(i)} := \bigoplus_{i=0}^{m-1} W_i^{(i)}$ (direct sum of subspaces).
  Define the multiplication on $A$ by $$\varphi^k(a)\varphi^\ell(b)=\binom{k+\ell}{k}_\zeta\ \varphi^{k+\ell}((c^\ell a)b) \text{ for all }a, b\in B \text{ and } 0 \leqslant k,\ell < m.$$ Then $A$ is an  $H_{m^2}(\zeta)$-simple algebra.
\end{theorem}
\begin{proof}
We check explicitly that the formulas indeed define on $A$ a structure of an $H_{m^2}(\zeta)$-module
algebra. 

Suppose that $I$ is an $H_{m^2}(\zeta)$-invariant ideal of $A$. Then $v^m I = 0$.
Let $t \in \mathbb Z_+$ such that $v^t I \ne 0$, $v^{t+1} I = 0$. Then $0 \ne v^t I \subseteq I \cap \ker v$. However, $\ker v = B$ is a simple graded algebra. Thus $\ker v \subseteq I$. Since $1_A \in I$,
we get $I = A$. Therefore, $A$ is an  $H_{m^2}(\zeta)$-simple algebra.
\end{proof}

Now we prove that we indeed have obtained all non-semisimple $H_{m^2}(\zeta)$-simple algebras.

\begin{theorem}\label{TheoremTaftSimpleNonSemiSimpleClassify}
Suppose $A$ is a finite dimensional $H_{m^2}(\zeta)$-simple algebra over a perfect field $F$ and $J:=J(A)\ne 0$.
Then $A$ is isomorphic to an algebra from Theorem~\ref{TheoremTaftSimpleNonSemiSimplePresent}.
\end{theorem}
\begin{corollary}
Let $A$ be a finite dimensional $H_{m^2}(\zeta)$-simple algebra over $F$
where $F$ is a field of characteristic $0$, an algebraically closed field, or a finite field.
Suppose $J:=J(A)\ne 0$.
Then $A$ is isomorphic to an algebra from Theorem~\ref{TheoremTaftSimpleNonSemiSimplePresent}.
\end{corollary}

In order to prove Theorem~\ref{TheoremTaftSimpleNonSemiSimpleClassify}, we need several auxiliary lemmas.

Let $M_1$ and $M_2$ be two $(A,A)$-graded bimodules for a $\mathbb Z_m$-graded algebra $A$.
We say that a linear isomorphism $\varphi \colon M_1 \to M_2$ is a \textit{$c$-isomorphism}
of $M_1$ and $M_2$ if there exists $r\in\mathbb Z$ such that $c\varphi(b) = \zeta^{-r} \varphi(cb)$, $\varphi(ab)=(c^r a)\varphi(b)$,
$\varphi(ba)=\varphi(b)a$ for all $b\in M_1$, $a\in A$.

\begin{lemma}\label{LemmaTaftSimpleNonSemiSimpleClassifySumDirect}
Suppose $A$ is a finite dimensional $H_{m^2}(\zeta)$-simple algebra over a field $F$ and $J:=J(A)\ne 0$.
Let $J^\ell = 0$, $J^{\ell-1} \ne 0$. Choose a minimal $\mathbb Z_m$-graded $A$-ideal $\tilde J \subseteq J^{\ell-1}$.
Then for any $k$, $J_k := \sum_{i=0}^{i=k} v^i \tilde J$ is a graded ideal of $A$ and $A = \bigoplus_{i=0}^t v^i \tilde J$ (direct sum of graded subspaces) for some $1 \leqslant t \leqslant m-1$.
Moreover, $J_k/J_{k-1}$, $0 \leqslant k \leqslant t$, are irreducible graded $(A,A)$-bimodules $c$-isomorphic to each other. (Here $J_{-1} := 0$.)
\end{lemma}
\begin{proof}
Since for any $a \in \tilde J$, $b\in A$, the element $(v^k a) b$ can be presented as a linear combination
of elements $v^i((c^{k-i} a)(v^{k-i}b))$, each $J_k := \sum_{i=0}^{i=k} v^i \tilde J$ is a graded ideal of $A$.

 Recall that $v^m =0$. Thus $J_m$ is an $H_{m^2}(\zeta)$-invariant ideal of $A$.
Hence  $A=J_m$.

Let $\varphi_k \colon J_k/J_{k-1} \to J_{k+1}/J_k$ where $0 \leqslant k \leqslant m-1$, be the map
defined by $\varphi_k (a + J_{k-1}) = va + J_k$. Then $c\varphi_k (\bar b) = \zeta^{-1}\varphi_k (c\bar b)$,
$$\varphi_k (a \bar b) = v(ab)+J_k = (ca)(vb)+(va)b+J_k = (ca)(vb)+J_k
=(ca)\varphi_k (\bar b),$$
$$\varphi_k (\bar b a) = v(ba)+J_k = (cb)(va)+(vb)a+J_k = (vb)a+J_k
=\varphi_k (\bar b) a$$
for all $a\in A$, $b \in J_k$.
Note that $\tilde J = J_0/J_{-1}$ is an irreducible graded $(A,A)$-bimodule.
Therefore, $J_{k+1}/J_k$ is an irreducible graded $(A,A)$-bimodule or zero for any $0 \leqslant k \leqslant m-1$. Thus if $A = J_t$, $A \ne J_{t-1}$,
then $\dim J_t = (t+1)\dim \tilde J$ and $A = \bigoplus_{i=0}^t v^i \tilde J$ (direct sum of graded subspaces).
\end{proof}

\begin{lemma}\label{LemmaTaftSimpleNonSemiSimpleClassifyUnity}
Suppose $A$ is a finite dimensional $H_{m^2}(\zeta)$-simple algebra over a perfect field $F$
where $J(A)\ne 0$.
Then $A$ has unity, $A/J(A)$ is a simple $\mathbb Z_m$-graded algebra,
and $J_{t-1} = J(A)$. (The ideal $J_{t-1}$ was defined in Lemma~\ref{LemmaTaftSimpleNonSemiSimpleClassifySumDirect}.)
\end{lemma}
\begin{proof} Note that $J(A)$ annihilates all irreducible $(A,A)$-bimodules. Thus 
 $J_k/J_{k-1}$
 are irreducible $(A/J(A),A/J(A))$-bimodules.
By~\cite{Taft}, there exists a maximal $\mathbb Z_m$-graded semisimple subalgebra $B\subseteq A$
such that $A = B \oplus J(A)$ (direct sum of $\mathbb Z_m$-graded subspaces), $B \cong A/J(A)$.
Note that $J_k/J_{k-1}$  are irreducible $(B,B)$-bimodules.
  Let $e$ be the unity of $B$. Then $$A = eAe \oplus (\id_A-e)A e \oplus e A(\id_A-e) \oplus (\id_A-e)A(\id_A-e)\text{ (direct sum of graded subspaces)}$$
 where $\id_A$ is the identity map.
 Note that $eAe$ is a completely reducible graded left $B \otimes  B^{\mathrm{op}}$-module,
 $(\id_A-e)A e$ is a completely reducible graded right $B$-module,
  $e A(\id_A-e)$ is a completely reducible  graded left $B$-module,
  and $(\id_A-e)A(\id_A-e)$ is a graded subspace with zero $B$-action.
  Thus $A$ is a sum of irreducible graded $(B,B)$-bimodules or bimodules with zero $B$-action. Since by Lemma~\ref{LemmaTaftSimpleNonSemiSimpleClassifySumDirect}, the algebra $A$ has a series of graded $(B,B)$-subbimodules with $c$-isomorphic irreducible factors, the only possibility is that $A=eAe$, $A/J(A)$ is a  simple graded algebra.
  Therefore, $J(A)$ is the unique maximal graded ideal.
Note that all $J_k/J_{k-1}\cong A/J(A)$ and, in particular, $A/J_{t-1}=J_t/J_{t-1} \cong A/J(A)$ (as vector spaces). Hence $\dim J_t = \dim J(A)$ and $J_{t-1} = J(A)$.  \end{proof}

\begin{lemma}\label{LemmaTaftSimpleNonSemiSimpleClassifyFormula}
Suppose $A$ is a finite dimensional $H_{m^2}(\zeta)$-simple algebra over a perfect field $F$
where $J(A)\ne 0$.
Define the linear map $\varphi \colon A \to A$ by $\varphi(v^k a) = v^{k-1} a$
for all $a \in\tilde J$, $1 \leqslant k \leqslant t$, $\varphi(\tilde J) = 0$.
(See Lemma~\ref{LemmaTaftSimpleNonSemiSimpleClassifySumDirect}.)
Then \begin{equation}\label{EqQuantumBinomPhi}
\varphi^k(a)\varphi^\ell(b)=\binom{k+\ell}{k}_\zeta\ \varphi^{k+\ell}((c^\ell a)b) \text{ for all }a, b\in \ker v \text{ and } 0 \leqslant k,\ell < m.
\end{equation}
\end{lemma}
\begin{proof}
Note that $\varphi(va)=a$ for all $a\in J_{t-1}$.
Thus the properties of $v$ imply $c\varphi(a) = \zeta \varphi(ca)$, $\varphi(ba)=(c^{-1}b)\varphi(a)$, 
$\varphi(ab)=\varphi(a)b$  for all $a\in v J_{t-1}$, $b\in \ker v$, and therefore for all $a\in A$, $b\in \ker v$ since $A = v J_{t-1} \oplus \tilde J$ (direct sum of graded subspaces) and $\varphi(\tilde J)=0$.
This proves~(\ref{EqQuantumBinomPhi}) for $k=0$ or $\ell=0$.

Recall that, by Lemma~\ref{LemmaTaftSimpleNonSemiSimpleClassifyUnity}, $\im \varphi = J_{t-1}=J(A)$.
Hence $(\im \varphi) \tilde J= \tilde J(\im \varphi) = 0$. Moreover $v\varphi(a)-a \in \tilde J$
for all $a\in A$. Now the case of arbitrary $k$ and $\ell$ is done by induction:
\begin{equation*}\begin{split}
\varphi^k(a)\varphi^\ell(b)= \varphi(v(\varphi^k(a)\varphi^\ell(b)))
=\varphi((c\varphi^k(a))\varphi^{\ell-1}(b)+\varphi^{k-1}(a)\varphi^\ell(b))=\\
\varphi\left((\zeta^k\varphi^k(ca)\varphi^{\ell-1}(b)+\varphi^{k-1}(a)\varphi^\ell(b)\right)=\\
\varphi\left(\zeta^k\binom{k+\ell-1}{k}_\zeta\ \varphi^{k+\ell-1}((c^\ell a)b)+\binom{k+\ell-1}{k-1}_\zeta\ \varphi^{k+\ell-1}((c^\ell a)b) \right)=\\
\left(\zeta^k \binom{k+\ell-1}{k}_\zeta + \binom{k+\ell-1}{k-1}_\zeta\right) \varphi^{k+\ell}((c^\ell a)b)
=\\ \binom{k+\ell}{k}_\zeta\ \varphi^{k+\ell}((c^\ell a)b)
\end{split}\end{equation*}
since \begin{equation*}\begin{split}
\zeta^k \binom{k+\ell-1}{k}_\zeta + \binom{k+\ell-1}{k-1}_\zeta
= \frac{\zeta^k (k+\ell-1)!_\zeta}{k!_\zeta (\ell-1)!_\zeta}+
\frac{(k+\ell-1)!_\zeta}{(k-1)!_\zeta \ell!_\zeta}=\\(\zeta^k \ell_\zeta + k_\zeta)\frac{(k+\ell-1)!_\zeta}{k!_\zeta \ell!_\zeta}=(k+\ell)_\zeta \frac{(k+\ell-1)!_\zeta}{k!_\zeta \ell!_\zeta}
= \frac{(k+\ell)!_\zeta}{k!_\zeta \ell!_\zeta}=\binom{k+\ell}{k}_\zeta.\end{split}\end{equation*}
\end{proof}

\begin{proof}[Proof of Theorem~\ref{TheoremTaftSimpleNonSemiSimpleClassify}.]
By Lemma~\ref{LemmaTaftSimpleNonSemiSimpleClassifyUnity}, there exists unity $1_A \in A$.
Note that $1_A \notin J_{t-1}$ (see the definition in Lemma~\ref{LemmaTaftSimpleNonSemiSimpleClassifySumDirect}),
since $J_{t-1}$ is an ideal. Hence $\varphi^t(1_A)\ne 0$.
(See the definition of the map $\varphi$ in Lemma~\ref{LemmaTaftSimpleNonSemiSimpleClassifyFormula}.)
Since $v\varphi(a)-a \in \tilde J$ for all $a\in A$,
we have $v\varphi^t(1_A)=\varphi^{t-1}(1_A)+j_1$
and $v\varphi(1_A)=1_A+j_2$ for some $j_1, j_2 \in \tilde J$.
Note that $\varphi^t(1_A) \varphi(1_A)=\binom{t+1}{t}_\zeta \varphi^{t+1}(1_A) = 0$.
However \begin{equation*}\begin{split}0=v(\varphi^t (1_A)\varphi(1_A))
=(v\varphi^t(1_A))\varphi(1_A)+(c\varphi^t(1_A))v\varphi(1_A)
=\\
(\varphi^{t-1}(1_A)+j_1)\varphi(1_A)+\zeta^t\varphi^t(1_A)(1_A+j_2)
=\varphi^{t-1}(1_A)\varphi(1_A)+\zeta^t\varphi^t(1_A)1_A
=\\ \left(\binom t{t-1}_\zeta+\zeta^t\right)\varphi^t(1_A)
=(t+1)_\zeta\ \varphi^t(1_A)\end{split}\end{equation*}
since by Lemma~\ref{LemmaTaftSimpleNonSemiSimpleClassifyUnity},
 $(\im \varphi)\tilde J = J_{t-1}\tilde J=J(A)\tilde J = 0$.
Hence $(t+1)_\zeta = 0$ and $t=m-1$.
By Lemma~\ref{LemmaTaftSimpleNonSemiSimpleClassifyUnity},
$J(A)=J_{m-2}$.
Thus $\ker v = v^{m-1}\tilde J \cong A/J(A)$.
Now~(\ref{EqQuantumBinomPhi}) implies the theorem.
\end{proof}

\begin{remark}
Since the maximal semisimple subalgebra $\ker v$ is uniquely determined, any two such $H_{m^2}(\zeta)$-simple algebras $A$ are isomorphic as $H_{m^2}(\zeta)$-module algebras if and only if their subalgebras $\ker v$ are isomorphic as $\mathbb Z_m$-algebras. Moreover, all automorphisms of $A$ as an $H_{m^2}(\zeta)$-module algebra are induced by the automorphisms of $\ker v$ as a $\mathbb Z_m$-algebra.
Indeed, let $\psi \colon A \to A$ be an automorphism of $A$ as an $H_{m^2}(\zeta)$-module algebra.
Since $\tilde J = J(A)^{m-1}$, $\psi(\tilde J) = \tilde J$ and $$v^{m-1}\psi(\varphi^{m-1}(a))=\psi(a)
\text{ for all }a\in \ker v$$ implies $$\psi(\varphi^{m-1}(a))=\varphi^{m-1}(\psi(a)).$$
Applying $v^{m-k-1}$, we get $\psi(\varphi^k(a))=\varphi^k(\psi(a))$
for all $a\in\ker v$ and $0 \leqslant k < m$.
\end{remark}


\section{Growth of polynomial $H_{m^2}(\zeta)$-identities}\label{SectionTaftSimpleAmitsur}

Here we apply the results of Section~\ref{SectionTaftSimpleNonSemisimple}
to polynomial $H_{m^2}(\zeta)$-identities.

First we introduce the notion of the free associative $H$-module algebra. Here we follow~\cite{BahturinLinchenko}.
Let $F \langle X \rangle$ be the free associative algebra without $1$
   on the set $X := \lbrace x_1, x_2, x_3, \ldots \rbrace$.
  Then $F \langle X \rangle = \bigoplus_{n=1}^\infty F \langle X \rangle^{(n)}$
  where $F \langle X \rangle^{(n)}$ is the linear span of all monomials of total degree $n$.
   Let $H$ be a Hopf algebra over $F$. Consider the algebra $$F \langle X | H\rangle
   :=  \bigoplus_{n=1}^\infty H^{{}\otimes n} \otimes F \langle X \rangle^{(n)}$$
   with the multiplication $(u_1 \otimes w_1)(u_2 \otimes w_2):=(u_1 \otimes u_2) \otimes w_1w_2$
   for all $u_1 \in  H^{{}\otimes j}$, $u_2 \in  H^{{}\otimes k}$,
   $w_1 \in F \langle X \rangle^{(j)}$, $w_2 \in F \langle X \rangle^{(k)}$.
We use the notation $$x^{h_1}_{i_1}
x^{h_2}_{i_2}\ldots x^{h_n}_{i_n} := (h_1 \otimes h_2 \otimes \ldots \otimes h_n) \otimes x_{i_1}
x_{i_2}\ldots x_{i_n}.$$ Here $h_1 \otimes h_2 \otimes \ldots \otimes h_n \in H^{{}\otimes n}$,
$x_{i_1} x_{i_2}\ldots x_{i_n} \in F \langle X \rangle^{(n)}$. 

Note that if $(\gamma_\beta)_{\beta \in \Lambda}$ is a basis in $H$, 
then $F\langle X | H \rangle$ is isomorphic to the free associative algebra over $F$ with free formal  generators $x_i^{\gamma_\beta}$, $\beta \in \Lambda$, $i \in \mathbb N$.
 We refer to the elements
 of $F\langle X | H \rangle$ as \textit{associative $H$-polynomials}.
 
In addition, $F \langle X | H \rangle$ becomes an $H$-module algebra with the $H$-action
defined by
$h(x^{h_1}_{i_1}
x^{h_2}_{i_2}\ldots x^{h_n}_{i_n})=x^{h_{(1)}{h_1}}_{i_1} x^{h_{(2)}{h_2}}_{i_2}\ldots
x^{h_{(n)}{h_n}}_{i_n}$ for $h\in H$.

Let $A$ be an associative $H$-module algebra.
Any map $\psi \colon X \to A$ has a unique homomorphic extension $\bar\psi
\colon F \langle X | H \rangle \to A$ such that $\bar\psi(h w)=h\psi(w)$
for all $w \in F \langle X | H \rangle$ and $h \in H$.
 An $H$-polynomial
 $f \in F\langle X | H \rangle$
 is an \textit{$H$-identity} of $A$ if $\varphi(f)=0$
for all homomorphisms $\varphi \colon F \langle X | H \rangle \to A$
of algebras and $H$-modules.
 In other words, $f(x_1, x_2, \ldots, x_n)$
 is an $H$-identity of $A$
if and only if $f(a_1, a_2, \ldots, a_n)=0$ for any $a_i \in A$.
 In this case we write $f \equiv 0$.
The set $\Id^{H}(A)$ of all $H$-identities
of $A$ is an $H$-invariant ideal of $F\langle X | H \rangle$.

We denote by $P^H_n$ the space of all multilinear $H$-polynomials
in $x_1, \ldots, x_n$, $n\in\mathbb N$, i.e.
$$P^{H}_n = \langle x^{h_1}_{\sigma(1)}
x^{h_2}_{\sigma(2)}\ldots x^{h_n}_{\sigma(n)}
\mid h_i \in H, \sigma\in S_n \rangle_F \subset F \langle X | H \rangle.$$
Then the number $c^H_n(A):=\dim\left(\frac{P^H_n}{P^H_n \cap \Id^H(A)}\right)$
is called the $n$th \textit{codimension of polynomial $H$-identities}
or the $n$th \textit{$H$-codimension} of $A$.

The analog of Amitsur's conjecture for $H$-codimensions can be formulated
as follows.

\begin{conjecture} There exists
 $\PIexp^H(A):=\lim\limits_{n\to\infty}
 \sqrt[n]{c^H_n(A)} \in \mathbb Z_+$.
\end{conjecture}

In the theorem below we consider the case $H=H_{m^2}(\zeta)$.

\begin{theorem}\label{TheoremTaftSimpleAmitsur} Let $A$ be a finite dimensional $H_{m^2}(\zeta)$-simple
algebra over an algebraically closed field $F$ of characteristic $0$.
Then there exist constants $C > 0$, $r\in \mathbb R$ such that $$C n^{r} (\dim A)^n \leqslant c^{H_{m^2}(\zeta)}_n(A) \leqslant (\dim A)^{n+1}\text{ for all }n \in \mathbb N.$$
\end{theorem}
\begin{corollary}
The analog of Amitsur's conjecture holds
 for such codimensions. In particular, $\PIexp^H(A)=\dim A$.
\end{corollary}

In order to prove Theorem~\ref{TheoremTaftSimpleAmitsur}, we need one lemma.

Let $k\ell \leqslant n$ where $k,n \in \mathbb N$ are some numbers.
 Denote by $Q^H_{\ell,k,n} \subseteq P^H_n$
the subspace spanned by all $H$-polynomials that are alternating in
$k$ disjoint subsets of variables $\{x^i_1, \ldots, x^i_\ell \}
\subseteq \lbrace x_1, x_2, \ldots, x_n\rbrace$, $1 \leqslant i \leqslant k$.

\begin{lemma}\label{LemmaTaftNSSAltHPolynomial}
Let $A$ be an $H_{m^2}(\zeta)$-simple non-semisimple associative algebra over an 
algebraically closed field $F$
of characteristic $0$, $\dim A=\ell m$.
Then there exists a number $n_0 \in \mathbb N$ such that for every $n\geqslant n_0$
there exist disjoint subsets $X_1$, \ldots, $X_k \subseteq \lbrace x_1, \ldots, x_n
\rbrace$, $k = \left[\frac{n-n_0}{2\ell m}\right]$,
$|X_1| = \ldots = |X_{k}|=\ell m$ and a polynomial $f \in P^{H_{m^2}(\zeta)}_n \backslash
\Id^{H_{m^2}(\zeta)}(A)$ alternating in the variables of each set $X_j$.
\end{lemma}
\begin{proof}
By Theorem~\ref{TheoremTaftSimpleNonSemiSimpleClassify}, $A=\bigoplus_{i=0}^{m-1} v^i \tilde J$ (direct sum of subspaces) where $\tilde J^2=0$ and $v^{m-1}\tilde J$ is a $\mathbb Z_m$-simple subalgebra.

Fix the basis $a_1, \ldots, a_\ell;\ va_1, \ldots, va_\ell; \ldots; v^{m-1}a_1, \ldots, v^{m-1}a_\ell$
in $A$ where $a_1, \ldots, a_\ell$ is a basis in $\tilde J$.

Since $v^{m-1}\tilde J$ is a $\mathbb Z_m$-simple subalgebra, by~\cite[Theorem~7]{ASGordienko3}, there exist $T \in \mathbb Z_+$ and  $\bar z_1, \ldots, \bar z_T \in v^{m-1}\tilde J$ such that
for any $k \in \mathbb N$
there exists $$f_0=f_0(x_1^1, \ldots, x_\ell^1; \ldots;
x^{2k}_1, \ldots,  x^{2km}_\ell;\ z_1, \ldots, z_T;\ z) \in Q^{F\mathbb Z_m}_{\ell, 2km, 2k\ell m+T+1}$$
such that for any $\bar z \in v^{m-1} \tilde J$ we have
$$f_0(v^{m-1} a_1, \ldots, v^{m-1} a_\ell; \ldots;
v^{m-1} a_1, \ldots, v^{m-1} a_\ell; \bar z_1, \ldots, \bar z_T; \bar z) = \bar z.$$

Take $n_0=T+1$, $k=\left[\frac{n-n_0}{2\ell m}\right]$, and consider \begin{equation*}\begin{split}f(x_1^1, \ldots, x_{\ell m}^1; \ldots;
x^{2k}_1, \ldots,  x^{2k}_{\ell m};\ z_1, \ldots, z_T; \ z;\ y_1, \ldots, y_{n-2k\ell m-T-1})=\\
\Alt_1 \Alt_2 \ldots \Alt_{2k}
f_0(x_1^1, \ldots, x_{\ell}^1;\ \left(x_{\ell+1}^1\right)^v, \ldots, \left(x_{2\ell}^1\right)^v;
\ \left(x_{2\ell+1}^1\right)^{v^2}, \ldots, \left(x_{3\ell}^1\right)^{v^2}; 
 \ldots;\\ \left(x_{\ell(m-1)+1}^1\right)^{v^{m-1}}, \ldots, \left(x_{\ell m}^1\right)^{v^{m-1}}; \ldots; \\
x_1^{2k}, \ldots, x_{\ell}^{2k};\ \left(x_{\ell+1}^{2k}\right)^v, \ldots, \left(x_{2\ell}^{2k}\right)^v;
\ \left(x_{2\ell+1}^{2k}\right)^{v^2}, \ldots, \left(x_{3\ell}^{2k}\right)^{v^2}; 
 \ldots;\\ \left(x_{\ell(m-1)+1}^{2k}\right)^{v^{m-1}}, \ldots, \left(x_{\ell m}^{2k}\right)^{v^{m-1}};\ z_1, \ldots, z_T;\ z)\ y_1 y_2 \ldots y_{n-2k\ell m-T-1}\in P^{H_{m^2}(\zeta)}_n\end{split}\end{equation*}
where $\Alt_i$ is the operator of alternation on the set $X_i:=\lbrace 
x_1^i, \ldots, x_{\ell m}^i\rbrace$.

Now we notice that \begin{equation*}\begin{split} f(v^{m-1}a_1, \ldots, v^{m-1}a_{\ell},\ \ldots\ , va_1, \ldots, va_{\ell},\ a_1, \ldots, a_{\ell};\ \ldots\ ; \\ v^{m-1} a_1, \ldots, v^{m-1} a_{\ell},\ \ldots\ ,
va_1, \ldots, va_{\ell},\ a_1, \ldots, a_{\ell};\ \bar z_1, \ldots, \bar z_T; \ 1_A, \ldots, 1_A)
=(\ell!)^{2k m} 1_A\end{split}\end{equation*} since $v^m=0$.
The lemma is proved.
\end{proof}

\begin{proof}[Proof of Theorem~\ref{TheoremTaftSimpleAmitsur}.]
If $A$ is semisimple, then Theorem~\ref{TheoremTaftSimpleAmitsur}
follows from~\cite[Theorem~5]{ASGordienko3}. If $A$ is not semisimple, we 
repeat verbatim the proof of~\cite[Lemma~11 and Theorem~5]{ASGordienko3}
using Lemma~\ref{LemmaTaftNSSAltHPolynomial} instead of~\cite[Lemma~10]{ASGordienko3}
and~\cite[Lemma~4]{ASGordienko3} instead of~\cite[Theorem~6]{ASGordienko3}.
\end{proof}

\section*{Acknowledgements}

I am grateful to E.~Jespers and E.~Aljadeff for helpful discussions.

\end{document}